\documentclass{amsart}
\usepackage{amssymb}
\usepackage{enumerate}
\usepackage{color}
\usepackage{mathtools}

\newtheorem{thm}{Theorem}[section]
\newtheorem{lem}[thm]{Lemma}
\newtheorem{defi}[thm]{Definition}
\newtheorem{prop}[thm]{Proposition}
\newtheorem{rem}[thm]{Remark}

\newtheorem{fact}[thm]{Scholium}
\newtheorem{cor}[thm]{Corollary}

\numberwithin{equation}{section}

\newcommand{\Zb}{\mathbb{Z}}
\newcommand{\Rb}{\mathbb{R}}

\newcommand{\Nb}{\mathbb{N}}
\newcommand{\Cb}{\mathbb{C}}

\newcommand{\Eb}{\mathbb{E}}

\newcommand{\A}{\mathcal{A}}

\newcommand{\X}{\mathcal{X}}

\newcommand\norm[1]{\big\| #1 \big\|}
\newcommand\md[1]{\left| #1 \right|}

\newcommand\p[1]{\left( #1 \right)}
\newcommand\set[1]{\left\lbrace #1 \right\rbrace}
\newcommand\floor[1]{\left\lfloor #1 \right\rfloor}

\newcommand{\e}{\varepsilon}

\newcommand{\Ind}{\mathbf{1}}

\newcommand\headd{\mathbin{\prec\mkern-3mu\prec_{\rm hd}}}
\newcommand\taill{\mathbin{\prec\mkern-3mu\prec_{\rm tl}}}
\newcommand\head{\prec_{\rm hd}}
\newcommand\tail{\prec_{\rm tl}}

\newcommand\supp{\text{\rm supp}}

\begin{document}

\title{Lorentz-Shimogaki-Arazy-Cwikel Theorem revisited}

\author[]{L. Cadilhac}
\address{Mathematical Institute of the Polish Academy of Science, ul. Sniadeckich 8
00-656 Warszawa}
\address{Normandie Univ,
  UNICAEN, CNRS, Laboratoire de Math{\'e}matiques Nicolas Oresme,
  14000 Caen, France}
\email{lcadilhac@impan.pl}

\author[]{F. Sukochev}
\address{School of Mathematics and Statistics, University of New South Wales, Kensington,  2052, Australia}
\email{f.sukochev@unsw.edu.au}

\author[]{D. Zanin}
\address{School of Mathematics and Statistics, University of New South Wales, Kensington,  2052, Australia}
\email{d.zanin@unsw.edu.au}

\begin{abstract} We present a new approach to Lorentz-Shimogaki and Arazy-Cwikel Theorems which covers all range of $p,q\in (0,\infty]$ for function spaces and sequence spaces. As a byproduct, we
solve a conjecture of Levitina and the last two authors.
\end{abstract}

\date{}
\maketitle

\section{Introduction}

Descriptions of interpolation spaces for couples of $L_p$-spaces, $1\leq p\leq \infty$ were extensively researched at the end of the 70's and in the 80's, providing satisfying answers to most problems which were considered relevant at the time.

However, new questions arising from noncommutative analysis recently highlighted some gaps in our knowledge of this subject, especially for the case of $p<1$. In this paper, we revisit some important results of the literature (\cite{LS71},\cite{AC84}, \cite{Spa78}), generalising them and thus filling some of the holes that were revealed in the theory. In particular, we answer a question asked in \cite{LSZ20} by Levitina and the last two authors and already partially studied in \cite{CN17} regarding the interpolation theory of sequence spaces (see Theorem \ref{thm:conj LSZ}). Besides this new result, this paper introduces a general approach which covers the range of all $0\leq p\leq\infty$ and is self-contained. It puts an emphasis on the use of the space $L_0$ of all finitely supported measurable functions.
As far as the authors know this space rarely appears in interpolation theory (however, see  \cite{HM90}, \cite{Astashkin-1994} and \cite{HS19}). We provide evidence that $L_0$ is a suitable ``left endpoint'' on the interpolation scale of $L_p$-spaces, despite its possessing an atypical structure (it is not even an $F$-space).

\smallskip

Recall that a function space $E$ is an {\bf interpolation space} for the couple $(L_p,L_q)$ if any linear operator $T$ bounded on $L_p$ and $L_q$ is also bounded on $E$ (see Definition \ref{def:interpolation space}). This notion provides a way of transfering inequalities well-known in $L_p$-spaces to more exotic ones. To both understand the range of applicability of this technique and be able to check whether it applies to a given function space $E$, we are interested in simple descriptions of interpolation spaces for the couple $(L_p,L_q)$.

This problem has a long history starting with  seminal Calder\'on-Mityagin theorem \cite{Mityagin65} and \cite{Cal66} on the couple $(L_1,L_\infty)$ and followed by Lorentz and Shimogaki's \cite{LS71} results on the couples $(L_1,L_q)$ and $(L_p,L_\infty)$ with $1\leq p, q\leq \infty$, which are stated in terms of various submajorizations, see also \cite[Theorem 7.2]{KM03}.

We consider two orders: {\bf head majorization} and {\bf tail majorization}. Head majorization coincides with the usual notion of (sub)majorization and already appears in Calder\'on's work. It is defined on $L_1 + L_{\infty}$ by:
$$g\headd f \Leftrightarrow \forall t>0,
\int_0^t\mu(s,g)ds\leq\int_0^t\mu(s,f)ds.$$
If moreover $f,g\in L_1$ and $\norm{f}_1 = \norm{g}_1$, then we write $g \head f$. Above and in the remainder of the text $\mu(g) : t \to \mu(t,g)$ denotes the right-continuous decreasing rearangement of $g$. Note that $\mu(g)$ is well-defined if and only if $g$ belongs to $L_0 + L_\infty$.

Tail majorization is defined on $L_0 + L_1$ by:
$$g\taill f \Leftrightarrow \forall t>0,
\int_t^\infty \mu(s,g)ds\leq\int_t^\infty \mu(s,f)ds.$$
If moreover $f,g\in L_1$ and $\norm{f}_1 = \norm{g}_1$, then we write $g \tail f$. Remark that $g\tail f$ if and only if $f \head g$.

It is well-known that the order $\headd$ is strongly linked to the interpolation theory of the couples $(L_p,L_\infty)$. We show that  similarly, the order $\taill$ is linked to the couples $(L_0,L_q)$. Combining these two tools, we recover in a self-contained manner characterisations of interpolation spaces for couples of arbitrary $L_p$-spaces, $0< p\leq \infty$ earlier obtained in \cite{Cad19}. Note that tail majorization coincides with the weak supermajorization of \cite{DDS14}.

Let $\mathcal{X}$ be the linear space of all measurable functions.
If not precised otherwise, {\bf the underlying measure space we are working on is $(0,\infty)$} equipped with the Lebesgue measure $m$.
We obtain the following:

\begin{thm}\label{thm : intro inter}
Let $E\subset\mathcal{X}$ be a quasi-Banach function space (a priori, not necessarily symmetric). Let $p,q\in (0,\infty)$ such that $p<q$. Then:
\begin{enumerate}[{\rm (a)}]
\item $E$ is an interpolation space for the couple $(L_p,L_\infty)$ if and only if there exists $c_{p,E}>0$ such that for any $f\in E$ and $g\in L_p+L_\infty$,
$$\md{g}^p \headd \md{f}^p \Rightarrow g\in E\ \text{and}\ \norm{g}_E \leq c_{p,E} \norm{f}_E;$$
\item  $E$ is an interpolation space for the couple $(L_0,L_q)$ if and only if there exists $c_{q,E}>0$ such that for any $f\in E$ and $g\in L_0+L_q$,
$$\md{g}^q \taill \md{f}^q \Rightarrow g\in E\ \text{and}\ \norm{g}_E \leq c_{q,E} \norm{f}_E;$$
\item $E$ is an interpolation space for the couples $(L_0,L_q)$ and $(L_p,L_\infty)$ if and only if it is an interpolation space for the couple $(L_p,L_q)$.
\end{enumerate}
\end{thm}

This extends results of Lorentz-Shimogaki and Arazy-Cwikel to the quasi-Banach setting and contributes to the two first questions asked by Arazy in \cite[p.232]{CP84} in the particular case of $L_p$-spaces, $0<p<\infty$. As mentioned before, our approach places $L_0$ as a left endpoint on the interpolation scale of $L_p$-spaces, in sharp contrast to earlier results which focused mostly on Banach spaces and had $L_1$ playing this part.
An advantage of our approach is that it naturally encompasses every symmetric quasi-Banach space since they are all interpolation spaces for the couple $(L_0,L_\infty)$ (see \cite{HM90}, \cite{Astashkin-1994}). On the contrary, there exist some symmetric Banach spaces which are {\it not} interpolation spaces for the couple $(L_1,L_\infty)$ (see \cite{SSS15}). This led to some difficulties which were customarily be circumvented with the help of various technical conditions such as the Fatou property (as appears for example in \cite{AM04}).

Our strategy in this paper is totally different from the techniques used in \cite{Mityagin65,Cal66, LS71, Spa78, Cwi81,  HM90, AC84, CP84, Astashkin-1994, CN17,  KM03, AM04, Cad19, Astashkin20} and is based on partition lemmas, which were originally developed in a deep paper due to Braverman and Mekler \cite{BM}, which lies outside of the realm of interpolation theory. The approach of Braverman and Mekler was subsequently revised and redeveloped in \cite{SZ09} and precisely this revision consitutes the core of our approach in this paper.

We restate partition lemmas based on \cite[Proposition 19]{SZ09} in Section \ref{bi-contraction}. These lemmas allow us to restrict head and tail majorizations to very simple situations and reduce the problem to functions taking at most two values. Then, we deduce interpolation results from those structural lemmas.

Note that this scheme of proof is quite direct and in particular, does not involve at any point duality related arguments which are applicable only to Banach spaces (\cite{LS71}) or more generally to $L$-convex quasi-Banach spaces (\cite{Kal84},\cite{Pop83}).

In Section \ref{sequence}, we pursue the same type of investigation, but in the setting of sequence spaces. The non-diffuse aspect of the underlying measure generates substantial technical difficulties. In particular, we require a new partition lemma which is not as efficient as those in Section \ref{bi-contraction} (compare Lemmas \ref{sequence operator lemma 1} and \ref{third operator lemma}). This deficiency has been first pointed out to the authors by Cwikel. However, we are still able to resolve  the conjecture of \cite{LSZ20} (in the affirmative) by combining Lemma \ref{sequence partition lemma 1} with a Boyd-type argument which we borrow from Montgomery-Smith \cite{Mon96}. In particular, we substantially strengthen the results in \cite{CN17}. Here is the precise statement that we obtain:

\begin{thm}\label{thm:conj LSZ}
Let $E \subset \ell_\infty$ be a quasi-Banach sequence space and $q\geq 1$. The following conditions are equivalent:
\begin{enumerate}[{\rm (a)}]
\item there exists $p<q$ such that $E$ is an interpolation space for the couple $(\ell^p,\ell^q)$;
\item there exists $c>0$ such that for any $u\in E$ and $v\in \ell_\infty$
 $$|v|^q \taill |u|^q \Rightarrow v\in E\ \text{and}\ \|v\|_E\leq c\|u\|_E;$$

 \item for any $u\in E$ and $v\in \ell_\infty$,
 $$|v|^q \taill |u|^q \Rightarrow v\in E.$$

\end{enumerate}
\end{thm}

We freely use results of Cwikel \cite{Cwi81} and the first author \cite{Cad19} to avoid repeating too many similar arguments.

\section*{Preliminaries}

\subsection*{Interpolation spaces}

The reader is referred to \cite{BL76} for more details on interpolation theory and \cite{KPS82} for an introduction to symmetric spaces. In the remainder of this section, $p$ and $q$ will denote two nonnegative reals such that $p\leq q$.

Let $(\Omega,m)$ be any measure space (in particular the following definitions apply to $\Nb$ equipped with the counting measure {\it i.e} sequence spaces).  As previously mentioned, $L_0(\Omega) \subset\mathcal{X}(\Omega)$ denotes the set of functions whose supports have finite measures, it is naturally equipped with the group norm $$\norm{f}_{0} = m( \supp\ f),\ f\in L_0(\Omega).$$
The \lq\lq norm\rq\rq\ of a linear operator  $T:L_0(\Omega)\to L_0(\Omega),$ is defined as follows:
$$\|T\|_{L_0\to L_0}=\sup_{f\in L_0}\frac{m({\rm supp}(Tf))}{m({\rm supp}(f))}.$$

\begin{defi} A linear space $E \subset \X(\Omega)$ becomes a quasi-Banach function space when equipped with a complete quasi-norm $\norm{.}_E$ such that
\begin{itemize}
\item if $f\in E$ and $g\in\mathcal{X}(\Omega)$ are such that $\md{g} \leq \md{f},$ then $g\in E$ and $\norm{g}_E \leq \norm{f}_E$.
\end{itemize}
\end{defi}

\begin{defi} A quasi-Banach function space $E\subset\mathcal{X}(\Omega)$ is called symmetric if
\begin{itemize}
\item if $f\in E$ and $g\in\mathcal{X}(\Omega)$ are such that $\mu(f) = \mu(g)$ then $g\in E$ and $\norm{g}_E = \norm{f}_E$;
\end{itemize}
\end{defi}

\begin{defi}[Bounded operator on a couple of quasi-Banach function spaces] Let $X$ and $Y$ be quasi-Banach function spaces. We say that a linear operator $T$ is bounded on $(X,Y)$ if $T$ is defined from $X+Y$ to $X+Y$ and restricts to a bounded operator from $X$ to $X$ and from $Y$ to $Y$. Set:
$$\norm{T}_{(X,Y)\to (X,Y)} = \max \p{\norm{T}_{X\to X},\norm{T}_{Y\to Y}}.$$
\end{defi}

Les us now recall the precise abstract definition of an interpolation space (see \cite{BL76}, \cite{KPS82}).

\begin{defi}[Interpolation space between quasi-Banach function spaces]\label{def:interpolation space} Let $X,$ $Y$ and $Z$ be quasi-Banach function spaces on $\Omega.$ We say that $Z$ is an interpolation space for the couple $(X,Y)$ if $X\cap Y \subset Z \subset X+Y$ and any bounded operator on $(X,Y)$ restricts to a bounded operator on $Z$. Denote by ${\rm Int}(X,Y)$ the set of interpolation spaces for the couple $(X,Y).$
\end{defi}

\begin{thm}\label{rem:def interpolation}
Let $\mathcal{V}$ be a separable topological linear space and let $A,B,C\subset\mathcal{V}$ be quasi-Banach spaces. If $C$ is an interpolation space for the couple $(A,B),$ then there exists a constant $c(A,B,C) >0$ such that for any bounded operator $T$ on $(A,B)$,
$$\|T\|_{C\to C}\leq c(A,B,C)\cdot\max\{\|T\|_{A\to A},\|T\|_{B\to B}\}.$$
\end{thm}
\begin{proof} In \cite[Lemma I.4.3]{KPS82}, the assertion is proved for Banach spaces. The argument for quasi-Banach spaces is identical (because it relies on the closed graph theorem, which holds for $F$-spaces, hence, for quasi-Banach spaces).
\end{proof}

If, in particular, $0<p<q<\infty$ and if $E$ is an interpolation space for the couple $(L_p(\Omega),L_q(\Omega)),$ then there exists a constant $c>0$ called interpolation constant of $E$ for $(L_p(\Omega),L_q(\Omega))$ such that for any bounded operator $T$ on $(L_p(\Omega),L_q(\Omega))$,
$$\norm{T}_{E\to E} \leq c\norm{T}_{(L_p(\Omega),L_q(\Omega)) \to (L_p(\Omega),L_q(\Omega))}.$$

Finally, let us recall the definition of the $K$-functional associated to a couple of quasi-Banach spaces. For any $t>0$, and $f\in X+Y$, let
$$K_t(f,X,Y) := \inf_{g+h = f} \norm{g}_X + t\norm{h}_Y.$$

In addition to that, we need to define an interpolation space between $L_0$ and a quasi-Banach space.

\begin{defi}[Bounded operator on a couple $(L_0(\Omega),Y)$ for a quasi-Banach function space $Y$] Let $Y$ be a quasi-Banach function space. We say that a linear operator $T$ is bounded on $(L_0(\Omega),Y)$ if $T$ is defined from $L_0(\Omega)+Y$ to $L_0(\Omega)+Y$ and restricts to a bounded operator from $L_0(\Omega)$ to $L_0(\Omega)$ and from $Y$ to $Y.$
\end{defi}

\begin{defi}[Interpolation space for a couple $(L_0(\Omega),Y)$ for a quasi-Banach function space $Y$]\label{def:interpolation space for L_0 and Y} Let $Y$ and $Z$ be quasi-Banach function spaces on $\Omega.$ We say that $Z$ is an interpolation space for the couple $(L_0(\Omega),Y)$ if $L_0(\Omega)\cap Y \subset Z \subset L_0(\Omega)+Y$ and any bounded operator on $(L_0(\Omega),Y)$ restricts to a bounded operator on $Z.$ Denote by ${\rm Int}(L_0(\Omega),Y)$ the set of interpolation spaces for the couple $(L_0(\Omega),Y).$
\end{defi}

%

\subsection*{Symmetry of interpolation spaces}

In this subsection, we show that a quasi-Banach interpolation space for a couple of symmetric spaces can always be remormed into a symmetric space. Note that similar results can be found in the literature, see for example \cite[Theorem 2.1]{KPS82}.

As usual, we will use the term {\bf measure preserving} for a map $\omega$ between measure spaces $(\Omega_1,\A_1,m_1)$ and $(\Omega_2,\A_2,m_2)$ verifying,
$$\forall A\in \A_1, \omega^{-1}(A) \in \A_2\ \text{and}\ m_2(\omega^{-1}(A)) = m_1(A).$$

\begin{lem}\label{lem: construct omega}
Assume that $\Omega$ is $(0,1)$, $(0,\infty)$ or $\Nb$. Let $0\leq f,g\in L_0(\Omega) + L_\infty(\Omega)$ and let $\varepsilon>0.$ Assume that $\mu(f) = \mu(g)$. There exists a measure preserving map $\omega:{\rm supp}(g)\to \supp(f)$ such that $(1+\varepsilon)(f\circ \omega)\geq\mu(f).$
\end{lem}

\begin{proof} 
{\bf Case 1:} Suppose first that $\mu(\infty,f) = \mu(\infty,g) =0.$

Define, for any $n\in\mathbb{Z},$
$$F_n = \big\{t:(1+\e)^n<f(t)\leq (1+\e)^{n+1}\big\},\quad G_n = \big\{t:(1+\e)^n<g\leq (1+\e)^{n+1}\big\}.$$
By assumption, $m(F_n)=m(G_n)$ for every $n\in\mathbb{Z}.$ Let $\omega_n:G_n\to F_n$ be an arbitrary measure preserving bijection.

\begin{sloppypar}
Define measure preserving map $\omega:{\rm supp}(g)\to{\rm supp}(f)$ by concatenating  \mbox{$\omega_n:G_n\to F_n,$ $n\in\mathbb{Z}.$} For every $t\in G_n,$ we have
$$f(\omega(t))\geq (1+\e)^n,\quad g\leq (1+\e)^{n+1}.$$
Thus,
$$(1+\e)f(\omega(t))\geq g,\quad t\in{\rm supp}(g).$$
This completes the proof of Step 1.
\end{sloppypar}

{\bf Case 2.} Let $\delta$ such that $(1+\delta)^2 = (1+\e)$. Let $a = \mu(\infty,f) = \mu(\infty,g) >0$. Define for any $n\geq 1$
$$F_n = \big\{t:a(1+\delta)^n<f(t)\leq a(1+\delta)^{n+1}\big\},\quad G_n = \big\{t:a(1+\delta)^n<g\leq a(1+\delta)^{n+1}\big\},$$
and
$$F_0 = \big\{t:(1+\delta)^{-1}a \leq f(t)\leq (1+\delta)a\big\},\quad G_0 = \big\{t:0 < g\leq (1+\delta)a\big\}.$$
By asumption for any $n\geq 1$, $m(G_n) = m(F_n)$ and $m(G_0) = m(F_0) = \infty$. For any $n\geq 0$, choose a measure preserving bijection $\omega_n$ from $G_n$ to $F_n$.

Define the measure preserving map $\omega : \supp(g) \to \supp(f)$ by concatenating the $\omega_n$'s. For any $n\geq 0$ and any $t \in G_n$,
$$f(\omega(t))\geq a(1+\delta)^{n-1},\quad g\leq a(1+\delta)^{n+1}.$$
Thus,
$$(1+\delta)^2f(\omega(t)) = (1+\e)f(\omega(t)) \geq g,\quad t\in{\rm supp}(g).$$
\end{proof}

\begin{lem}\label{interpolation implies symmetry} Assume that $\Omega$ is $(0,1)$, $(0,\infty)$ or $\Nb$. Let $E,A,B \subset (L_0+L_\infty)(\Omega)$ be quasi-Banach function spaces. Assume that $A$ and $B$ are symmetric and that $E$ is an interpolation space for the couple $(A,B)$. Then $E$ admits an equivalent symmetric quasi-norm.
\end{lem}
\begin{proof}
Let $f\in E$ and $g\in L_0+L_\infty$. Assume that $\mu(g) \leq \mu(f)$. By Lemma \ref{lem: construct omega}, there exists a map $\omega : \supp(g) \to \supp(f)$ such that for any $t\in \supp(g)$,
$$2\md{f \circ \omega (t)} \geq \md{g(t)}.$$
Define, for any $h \in\X(\Omega)$,
$$
T(h) :=
\begin{cases}
\dfrac{g}{f\circ \omega} h \circ \omega &\text{on}\ \supp(g) \\
0 &\text{elsewhere.}
\end{cases}
$$
Since $\omega$ is measure preserving, $T$ is bounded on $A$ and $B$ of norm less than $2$. Let $c_E$ be the interpolation constant of $E$ for the couple $(A,B)$ (as in Theorem \ref{rem:def interpolation}). We know that $Tf = g \in E$ and
\begin{equation}\label{eq : symmetry}
\norm{g} \leq 2c_E\norm{f}.
\end{equation}
Define, for any $f\in E$
$$\norm{f}_{E'} = \inf_{\mu(g) \geq \mu(f)} \norm{g}_E.$$
By \eqref{eq : symmetry}, $\norm{f}_{E'} \leq \norm{f}_E \leq 2c_E \norm{f}_{E'}$ and $(E,\norm{\cdot}_{E'})$ is a symmetric space.
\end{proof}

\begin{rem}
\rm
It is not difficult to see that if the underlying measure space $\Omega$ contains both a continuous part and atoms, Lemma \ref{interpolation implies symmetry} is no longer true for $A = L_p(\Omega)$, $B = L_q(\Omega)$ and $p < 1$. However, one can observe that if $A$ and $B$ are {\it fully symmetric} ({\it i.e.} interpolation spaces between $L_1(\Omega)$ and $L_\infty(\Omega)$), Lemma \ref{interpolation implies symmetry} remains valid for any $\Omega$. This is reminiscent of the conditions required in \cite[Section 4]{Spa78}.
\end{rem}

\section{Interpolation for the couple $(L_0,L_q)$}

In this section, $\Omega=(0,\infty)$ (for brevity, we omit $\Omega$ in the notations). We investigate some basic properties of the interpolation couple $(L_0,L_q).$ First, we provide a statement analogous to Theorem \ref{rem:def interpolation} and applicable to $L_0$. 

Since the closed graph theorem does not apply to $L_0$ (it is not an $F$-space), our proof uses a concrete constructions that relies on the structure of the underlying measure space.

For any $f\in \X$, denote by $M_f$ the multiplication operator $g \mapsto f\cdot g$.

\begin{thm}\label{thm: inter L_0}
Let $E$ be a quasi-Banach function space and $q\in (0,\infty]$. Assume that $E$ is an interpolation space for the couple $(L_0,L_q)$. Then, there exists a constant $c$ such that for any contraction $T$ on $(L_0,L_q)$, $\norm{T}_{E\to E} \leq c$.
\end{thm}
\begin{proof} Let $(A_n)_{n\geq1}$ be a partition of $(0,\infty)$ such that $m(A_n)=\infty$ for every $n\geq1.$
       
Let $\gamma_n:A_n\to A_n^c$ be a measure preserving bijective transform. Set
$$(U_nx)(t)=
\begin{cases}
x(\gamma_n(t)),& t\in A_n\\
0,&t\in A_n^c
\end{cases},\quad
(V_nx)(t)=
\begin{cases}
x(\gamma_n^{-1}(t)),& t\in A_n^c\\
0,&t\in A_n
\end{cases}.
$$
Obviously, $U_n$ and $V_n$ are bounded operators on the couple $(L_0,L_q).$ By assumption, $U_n,V_n:E\to E$ are bounded mappings.

Note that
$$V_nU_n=M_{\chi_{A_n^c}},\quad n\geq1.$$

Let us argue by contradiction. For any $n\geq 1$, choose an operator $T_n$ which is a contraction on $(L_0,L_q)$ and such that
\begin{equation}\label{intl0 eq0}
\|T_n\|_{E\to E}\geq 4^n\cdot \max\{\|U_n\|_{E\to E},\|V_n\|_{E\to E},1\}^2.
\end{equation}

It is immediate that
$$T_n=M_{\chi_{A_n}}T_nM_{\chi_{A_n}}+M_{\chi_{A_n^c}}T_nM_{\chi_{A_n}}+M_{\chi_{A_n}}T_nM_{\chi_{A_n^c}}+M_{\chi_{A_n^c}}T_nM_{\chi_{A_n^c}}=$$
$$=T_{1,n}+V_nT_{2,n}+T_{3,n}U_n+V_nT_{4,n}U_n,$$
where
\begin{align*}
T_{1,n} = M_{\chi_{A_n}}T_nM_{\chi_{A_n}},
&\quad T_{2,n}= U_nT_nM_{\chi_{A_n}},\\
\quad T_{3,n}=M_{\chi_{A_n}}T_nV_n,&\quad T_{4,n}=U_nT_nV_n.
\end{align*}
By quasi-triangle inequality, we have
$$\|T_n\|_{E\to E}\leq C_E^2\cdot\Big(\sum_{k=1}^4\|T_{k,n}\|_{E\to E}\Big)\cdot\max\{\|U_n\|_{E\to E},\|V_n\|_{E\to E}.1\}^2.$$
Let $k_n\in\{1,2,3,4\}$ be such that
$$\|T_{k_n,n}\|_{E\to E}=\max_{1\leq k\leq 4}\|T_{k,n}\|_{E\to E}.$$
We, therefore, have
\begin{equation}\label{intl0 eq1}
\|T_n\|_{E\to E}\leq 4C_E^2\|T_{k_n,n}\|_{E\to E}\cdot\max\{\|U_n\|_{E\to E},\|V_n\|_{E\to E}.1\}^2.
\end{equation}

Set $S_n=T_{k_n,n}.$ Note that $\|S_n\|_{L_0\to L_0}\leq 1,$ $\|S_n\|_{L_q\to L_q}\leq 1.$ A combination of \eqref{intl0 eq0} and \eqref{intl0 eq1} yields
$$\|S_n\|_{E\to E}\geq 4^{n-1}C_E^{-2},\quad n\geq1.$$
       
Note that
$$S_n=M_{\chi_{A_n}}S_nM_{\chi_{A_n}}.$$
Set
$$S=\sum_{n\geq1}S_n.$$
\begin{sloppypar}

Since the $S_n$'s are in direct sum, \mbox{$\|S\|_{L_0\to L_0} = \sup_{n\geq 1} \norm{S_n}_{L_0 \to L_0} \leq 1$} and \mbox{$\|S\|_{L_q\to L_q} = \sup_{n\geq 1} \norm{S_n}_{L_q\to L_q} \leq 1.$}

Moreover, $E$ is an interpolation space for the couple $(L_0,L_q),$ it follows that $S:E\to E$ is bounded.

For any $n\geq 1$, choose $f_n\in E$ such that \mbox{$\|f_n\|_E\leq 1$} and \mbox{$\|S_nf_n\|_E\geq 4^{n-2}C_E^{-2}.$} Recall that $S_n=S_nM_{\chi_{A_n}}.$ Hence, we may assume without loss of generality that $f_n$ is supported on $A_n.$ Thus, $S(f_n)=S_n(f_n)$ and
$$\norm{S(f_n)}_E = \norm{S_n(f_n)}_E \geq 4^{n-2}C_E^{-2}.$$
This contradicts the boundedness of $S.$
\end{sloppypar}
\end{proof}

\begin{rem}
{\rm
Theorem \ref{thm: inter L_0} above remains true for other underlying measure spaces:
\begin{itemize}
\item {\it for sequence spaces}. Indeed, in the proof of Theorem \ref{thm: inter L_0}, we only use properties of the underlying measure space in the first sentence, namely when we consider a partition of $(0,\infty)$ into countably many sets, each of them isomorphic to $(0,\infty)$. Since a partition satisfying the same property exists for $\Zb_+$, Theorem \ref{thm: inter L_0} remains true for interpolation spaces between $\ell_0$ and $\ell_q$.
\item {\it for $(0,1)$}. The same general idea applies in this case but some modification have to be made because the maps $\gamma_n$ introduced in the proof cannot be assumed to be measure preserving. The details are left to the reader.
\end{itemize}
}
\end{rem}

\begin{lem}\label{interpolation implies symmetry second lemma} Let $E,Y\subset L_0+L_\infty$ be quasi-Banach function spaces. Assume that $Y$ is symmetric and that $E$ is an interpolation space for the couple $(L_0,Y).$  Then $E$ admits an equivalent symmetric quasi-norm.
\end{lem}
\begin{proof} The argument follows that in Lemma \ref{interpolation implies symmetry} {\it mutatis mutandi}.
\end{proof}

The following assertion is a special case Theorem \ref{thm : intro inter} and an important ingredient in the proof of the latter theorem.

\begin{thm}\label{cor: taill to int}
Let $E$ be a quasi-Banach function space and $q\in (0,\infty)$. Assume that $L_0\cap L_q\subset E\subset L_0+ L_q$ and that for any $f\in E$ and $g\in L_0+L_q,$
$$\md{g}^q \taill \md{f}^q \Rightarrow g\in E.$$
Then $E$ is an interpolation space for the couple $(L_0,L_q).$
\end{thm}

The rest of this section is occupied with the proof of Theorem \ref{cor: taill to int}.

Let $C_E$ be the concavity modulus of $E,$ that is, (the minimal) constant such that
$$\|x_1+x_2\|_E\leq C_E(\|x_1\|_E+\|x_2\|_E),\quad x_1,x_2\in E.$$
It allows to write a triangle inequality with infinitely many summands. This inequality should be understood in the usual sense: if scalar-valued series in the right hand side converges, then the series on the left hand side converges in $E$ and the inequality holds.

\begin{lem}\label{infinite triangle inequality} Let $E$ be a quasi-Banach space. We have
$$\|\sum_{n\geq1}x_n\|_E\leq\sum_{n\geq1}C_E^n\|x_n\|_E,\quad (x_n)_{n\geq1}\subset E.$$
\end{lem}
\begin{proof} The proof is standard and is, therefore, omitted.
\end{proof}

\begin{lem}\label{hvost} Let $E\subset L_0+L_q$ be a symmetric quasi-Banach function space. If $\|f\|_E\leq 1,$ then $\|\mu(f)\chi_{(1,\infty)}\|_q\leq c_E.$
\end{lem}
\begin{proof} Assume the contrary. Choose $(f_n)_{n\geq1}\in E$ such that $\|f_n\|_E\leq 1$ and $\|\mu(f_n)\chi_{(1,\infty)}\|_q\geq n.$ Set
$$h_n=\sum_{k\geq0}\mu(k+1,f_n)\chi_{(k,k+1)},\quad n\geq1.$$
It is immediate that $\|h_n\|_E\leq 1$ and $\|h_n\|_q\geq n.$

Let $C_E$ be the concavity modulus of the space $E$ and let $m\geq 4C_E$ be a natural number. Set
$$h=\sum_{n\geq1}(2C_E)^{-n}h_{m^n}.$$
By Lemma \ref{infinite triangle inequality}, we have
$$\|h\|_E\leq \sum_{n\geq1}C_E^n\|(2C_E)^{-n}h_{m^n}\|_E\leq\sum_{n\geq1}2^{-n}=1.$$
Since $h\in E,$ it follows that $h\in L_0+L_q.$ Since $h=\mu(h)$ is constant on the interval $(0,1),$ it follows that $h\in L_q.$ Thus,
$$\|h\|_q\geq (2C_E)^{-n}\|h_{m^n}\|_q\geq (2C_E)^{-n}\cdot m^n\geq 2^n,\quad n\geq1.$$
It follows that $h\notin L_q.$ This contradiction completes the proof.
\end{proof}

\begin{lem}\label{l0lq bicontraction lemma} Let $E$ be as in Theorem \ref{cor: taill to int}. If $T$ is a contraction on $(L_0,L_q),$ then $T:E\to E.$
\end{lem}
\begin{proof} Take $f\in E.$ Fix $t>0$ and let $A=A_1\bigcup A_2,$ where
$$A_1=\{|f|>\mu(t,f)\}\mbox{ and }A_2\subset\{|f|=\mu(t,f)\}\mbox{ is such that }m(A_2)=t-m(A_1).$$    
Obviously, $m(A)=t$ and
$$\mu(s,f\chi_{A^c})=\mu(s+t,f),\quad s>0.$$
Now, we write
$$Tf=T(f\chi_A)+T(f\chi_{A^c}).$$
Since $T:L_q\to L_q$ is a contraction, it follows that
$$\|T(f\chi_{A^c})\|_q^q\leq\|f\chi_{A^c}\|_q^q=\int_t^{\infty}\mu^q(s+t,f)ds=\int_t^{\infty}\mu^q(s,f)ds,$$
$$\|T(f\chi_A)\|_0\leq\|f\chi_A\|_0\leq t.$$

Now,
$$\int_t^{\infty}\mu^q(s,Tf)ds=\inf\{\|Tf-h\|_q^q:\ \|h\|_0\leq t\}.$$
Taking $h=T(f\chi_A),$ we write
$$\int_t^{\infty}\mu^q(s,Tf)ds\leq \|Tf-T(f\chi_A)\|_q^q=\|T(f\chi_{A^c})\|_q^q\leq\int_t^{\infty}\mu^q(s,f)ds .$$

Since $t>0$ is arbitrary, it follows that
$$|Tf|^q\prec\prec_{{\rm tl}}|f|^q.$$
By assumption, $Tf\in E.$ Since $f\in E$ is arbitrary, it follows that $T:E\to E.$
\end{proof}

Our next lemma demonstrates the boundedness of the mapping $T:E\to E$ established in Lemma \ref{l0lq bicontraction lemma}.

\begin{lem}\label{closed graph lemma} Let $E$ be as in Theorem \ref{cor: taill to int}. If $T$ is a contraction on $(L_0,L_q),$ then the mapping $T:E\to E$ is bounded.
\end{lem}
\begin{proof} Using Lemma \ref{interpolation implies symmetry second lemma}, we may assume without loss of generality that $E$ is symmetric.

Let us show that the graph of $T:E\to E$ is closed. Assume the contrary. Choose a sequence $(f_n)_{n\geq1}\subset E$ such that $f_n\to f\in E$ and $Tf_n\to g\neq Tf$ in $E.$ Replacing $(f_n)_{n\geq1}$ with $(f_n-f)_{n\geq1}$ if necessary, we may assume without loss of generality that $f=0.$ That is, $f_n\to 0$ in $E$ and $Tf_n\to g\neq0$ in $E.$
       
Passing to a subsequence, if needed, we may assume without loss of generality that
$$\frac{\|f_n\|_E}{\|\chi_{(0,\frac1n)}\|_E}\to0,\quad n\to\infty.$$
Clearly,
$$\|f_n\|_E=\|\mu(f_n)\|_E\geq\|\mu(f_n)\chi_{(0,\frac1n)}\|_E\geq\mu(\frac1n,f_n)\|\chi_{(0,\frac1n)}\|_E.$$
It follows that $\mu(\frac1n,f_n)\to0$ as $n\to\infty.$

Set
$$f_{1n}=f_n\chi_{\{|f_n|>\mu(\frac1n,f_n)\}},\quad n\geq1,$$
$$f_{2n}=f_n\chi_{\{\mu(1,f_n)<|f_n|\leq \mu(\frac1n,f_n)\}},\quad n\geq 1,$$
$$f_{3n}=f_n\chi_{\{|f_n|\leq \mu(1,f_n)\}},\quad n\geq1.$$

Since $T:L_0\to L_0$ is a contraction, it follows that
$$\|T(f_{1n})\|_0\leq\|f_{1n}\|_0\leq\frac1n,\quad n\geq1.$$
In particular, $T(f_{1n})\to0$ in measure.

Since $T:L_q\to L_q$ is a contraction, it follows that
$$\|T(f_{2n})\|_q\leq\|f_{2n}\|_q.$$
Since $f_{2n}$ is supported on a set of measure $1,$ it follows that
$$\|f_{2n}\|_q\leq\|f_{2n}\|_{\infty}=\mu(\frac1n,f_n)\to0,\quad n\to\infty.$$
In particular, $T(f_{2n})\to0$ in measure.

Since $T:L_q\to L_q$ is a contraction, it follows that $$\|T(f_{3n})\|_q\leq\|f_{3n}\|_q\leq\|\min\{\mu(f_n),\mu(1,f_n)\}\|_q.$$
By Lemma \ref{hvost}, we have
$$\|\min\{\mu(f_n),\mu(1,f_n)\}\|_q\leq c_{q,E}\|f_n\|_E\to0,\quad n\to\infty.$$
In particular, $T(f_{3n})\to0$ in measure.

Combining the last $3$ paragraphs, we conclude that $Tf_n\to 0$ in measure. However, $Tf_n\to g$ in $E$ and, therefore, in measure. It follows that $g=0,$ which contradicts our choice of the sequence $(f_n)_{n\geq1}.$ This contradiction shows that the graph of $T:E\to E$ is closed. By the closed graph theorem, $T:E\to E$ is bounded. 
\end{proof}

\begin{proof}[Proof of Theorem \ref{cor: taill to int}] Let $T$ be a bounded operator on $(L_0,L_q).$ Fix $n\in\mathbb{N}$ such that
$$\|T\|_{L_0\to L_0}\leq n,\quad \|T\|_{L_q\to L_q}\leq n.$$

Let $\sigma_{\frac1n}$ be a dilation operator defined by the usual formula
$$(\sigma_{\frac1n}f)(t)=f(nt),\quad n\geq0.$$
Set
$$S=\frac1n \sigma_{\frac1n}\circ T\mbox{ so that }T=n\sigma_n\circ S.$$
Obviously,
$$\|S\|_{L_0\to L_0}\leq 1,\quad \|S\|_{L_q\to L_q}\leq 1.$$
By Lemma \ref{closed graph lemma}, $S:E\to E$ is a bounded mapping. Since $E$ is a symmetric quasi-Banach function space, it follows that $\sigma_n:E\to E$ is bounded; this implies $T:E\to E.$
\end{proof}

Theorem \ref{cor: taill to int} applies in particular to $L_p$-spaces, $p\leq q.$ We decided to add more precise statement and to provide a direct proof of the latter.

\begin{cor}\label{lem : Lp in Int L0 Lq}
Let $p,q\in (0,\infty)$ such that $p<q$. Then, $L_p$ is an interpolation space for the couple $(L_0,L_q).$ More precisely, if $T$ is a contraction on $(L_0,L_q),$ then $T$ is a contraction on $L_p$.
\end{cor}
\begin{proof}
Let us first consider characteristic functions. Let $E$ be a set with finite measure. Since $T$ is a contraction on $L_0$, the measure of the support of $T(\chi_E)$ is less than $m(E)$. So by H\"older's inequality, setting $r = (p^{-1} - q^{-1})^{-1}$, we have:
$$\|T(\chi_{E})\|_p\leq\|T(\chi_{E})\|_q\cdot m(E)^{1/r}\leq\|\chi_E\|_q\cdot m(E)^{\frac1r}=\|\chi_E\|_p.$$

First, consider the case $p\leq 1.$ Let $f\in L_p$ be a step function, i.e.
$$f = \sum_{i\in\Nb} a_i\chi_{E_i},$$
where $a_i\in \Cb$ and the sets $E_i$ are disjoint sets with finite measure. By the $p$-triangular inequality:
$$\|Tf\|_p^p\leq\sum_{i\in\Nb}|a_i|^p\|T(\chi_{E_i})\|_p^p \leq \sum_{i\in\Nb}|a_i|^p\|\chi_{E_i}\|_p^p = \|f\|_p^p.$$
Since $T:L_p\to L_p$ is bounded by Theorem \ref{cor: taill to int} and since step functions are dense in $L_p,$ it follows that $T:L_p\to L_p$ is a contraction (for $p\leq 1$).

Now consider the case $p>1.$ Since $p<q,$ it follows that $q>1.$ By the preceding paragraph, $T:L_1\to L_1$ is a contraction. By complex interpolation, $T:L_p\to L_p$ is also contraction.
\end{proof}

\section{Construction of contractions on ($L_0$,$L_q$) and ($L_p$,$L_\infty$)} \label{bi-contraction}

Let $p,q\in (0,\infty)$. In this section, we are interested in the following question. Given functions $f$ and $g$ in $L_0 + L_q$ (resp. $L_p+L_\infty$), does there exist a bi-contraction $T$ on ($L_0$,$L_q$) (resp. ($L_p$,$L_\infty$)) such that $T(f) = g$? We show that such an operator exists provided that $\md{g}^q \taill \md{f}^q$ (resp. $\md{g}^p \headd \md{f}^p.$ This directly implies a necessary condition for a symmetric space to be an interpolation space for the couple $(L_0,L_q)$ (resp. $(L_p,L_\infty)$) which will be exploited in the next section.

Our method of proof is very direct. We construct the bi-contraction $T$ as direct sums of very simple operators. This is made possible by three partition lemmas that enable us to understand the orders $\taill$ and $\headd$ as direct sums of simple situations.

\subsection{Partition lemmas}

We state our first lemma without proof since it essentially repeats that of Proposition 19 in \cite{SZ09}.

\begin{lem}\label{first partition lemma} Let $f,g\in L_1$ be positive decreasing step functions. Assume that $g\head f.$ There exists a partition $\{I_k,J_k\}_{k\geq0}$ of $(0,\infty)$ such that
\begin{enumerate}[{\rm (i)}]
\item\label{191} for every $k\geq0,$ $I_k$ and $J_k$ are disjoint intervals of finite length;
\item\label{192} $(I_k\cup J_k)\cap(I_l\cup J_l)=\varnothing$ for $k\neq l;$
\item\label{193} $f$ and $g$ are constant on $I_k$ and on $J_k;$
\item\label{194} $g|_{I_k\cup J_k}\head f|_{I_k\cup J_k}$ for every $k\geq0;$
\end{enumerate}
\end{lem}
\begin{rem}
{\rm
Note that in \cite{SZ09}, Proposition 19 is proved for couples of functions $f,g$ of the form
$$
f = \sum_{i\in\Zb} f_i\Ind_{[a_i,a_{i+1})},\ g = \sum_{i\in\Zb} g_i\Ind_{[a_i,a_{i+1}]},
$$
with $(a_i)$ an increasing sequence in $(0,\infty)$ and $f_i,g_i \in \Cb$ for any $i\in \Nb$. However, the proof applies with very little modification to couples of decreasing positive step functions as in Lemma \ref{first partition lemma} above.
}
\end{rem}

\begin{fact}\label{majorization fact} Let $f,g\in\mathcal{X}$ be positive decreasing functions. Let $\Delta\subset(0,\infty)$ be an arbitrary measurable set.
\begin{enumerate}[{\rm (i)}]
\item if $f,g\in L_1+L_{\infty}$ are such that
$$\int_{[0,t]\cap\Delta}g\leq \int_{[0,t]\cap\Delta}f,\quad t>0,$$
then $g\chi_{\Delta}\headd f\chi_{\Delta}.$
\item if $f,g\in L_0+L_1$ are such that
$$\int_{(t,\infty)\cap\Delta}g\leq \int_{(t,\infty)\cap\Delta}f,\quad t>0,$$
then $g\chi_{\Delta}\taill f\chi_{\Delta}.$
\end{enumerate}
\end{fact}

Our second partition lemma shows that the order $\headd$ can be reduced for our purpose to a direct sum of $\head$ and $\leq$.

\begin{lem}\label{second partition lemma} Let $f,g\in L_1+L_{\infty}$ be such that $f = \mu(f)$, $g=\mu(g)$ and $g\headd f.$ There exists a collection $\{\Delta_k\}_{k\geq0}$ of pairwise disjoint sets such that
\begin{enumerate}[{\rm (i)}]
\item $g|_{\Delta_k}\head f|_{\Delta_k}$ for every $k\geq0;$
\item $g\leq f$ on the complement of $\bigcup_{k\geq0}\Delta_k;$
\end{enumerate}
\end{lem}
\begin{proof} Consider the set $\{g>f\}.$ Since $g-f$ is right-continuous, it follows that, for every $t\in\{g>f\},$ there exists an $\epsilon>0$ (which depends on $t$) such that $[t,t+\epsilon)\subset\{g>f\}.$ Hence, connected components of the set $\{g>f\}$ are intervals (closed or not) not reduced to points. Let us enumerate these intervals as $(a_k,b_k),$ $k\geq0$ (it does not really matter for us if boundary points of these intervals belong to $\{g>f\}$).

We have
$$\int_0^t(f-g)_+-\int_0^t(f-g)_-=\int_0^t(f-g)\geq 0.$$
Let
$$H(t)=\inf\Big\{u:\ \int_0^u(f-g)_+=\int_0^t(f-g)_-\Big\}.$$
Obviously, $H$ is a monotone function, $H(t)\leq t$ for all $t>0$ and
$$\int_0^{H(t)}(f-g)_+=\int_0^t(f-g)_-.$$

Set
$$\Delta_k=(a_k,b_k)\cup\Big((H(a_k),H(b_k))\cap\{f\geq g\}\Big).$$

Note that
$$\int_0^{a_k}(f-g)_+=\int_0^{b_k}(f-g)_+\geq\int_0^{b_k}(f-g)_-$$
and, therefore, $H(b_k)\leq a_k.$

We claim that $\Delta_k\cap\Delta_l=\varnothing$ for $k\neq l.$ Indeed, let $a_k<b_k\leq a_l<b_l.$ We have $H(a_k)\leq H(b_k)\leq H(a_l)\leq H(b_l).$ Thus, $(H(a_k),H(b_k))\cap (H(a_l),H(b_l))=\varnothing.$ We now have
$$\Delta_k\cap\Delta_l=\Big((\Delta_k\cap\{f<g\})\cap(\Delta_l\cap\{f<g\})\Big)\bigcup\Big((\Delta_k\cap\{f\geq g\})\cap(\Delta_l\cap\{f\geq g\})\Big).$$
Obviously,
$$(\Delta_k\cap\{f<g\})\cap(\Delta_l\cap\{f<g\})=(a_k,b_k)\cap(a_l,b_l)=\varnothing,$$
$$(\Delta_k\cap\{f\geq g\})\cap(\Delta_l\cap\{f\geq g\})=(H(a_k),H(b_k))\cap (H(a_l),H(b_l))\cap\{f\geq g\}=\varnothing.$$
This proves the claim.

We now claim that
$$\int_{[0,t]\cap\Delta_k}(f-g)\geq0.$$
If $t\leq a_k,$ then, taking into account that $H(b_k)\leq a_k,$ we infer that $[0,t]\cap\Delta_k\subset\{f\geq g\}$ and the claim follows immediately. If $t\in(a_k,b_k),$ then
$$\int_{[0,t]\cap\Delta_k}(f-g)=\int_{(H(a_k),H(b_k))}(f-g)_+-\int_{a_k}^t(f-g)_-\geq$$
$$\geq \int_{(H(a_k),H(b_k))}(f-g)_+-\int_{a_k}^{b_k}(f-g)_-=0.$$
This proves the claim.

It follows from the claim and Scholium \ref{majorization fact} that $g\chi_{\Delta_k}\headd f\chi_{\Delta_k}.$ Since
$$\int_{\Delta_k}g=\int_{\Delta_k}f,$$
the first assertion follows.

By construction, $(a_k,b_k)\subset\Delta_k.$ Thus,
$$\{g>f\}=\bigcup_{k\geq0}(a_k,b_k)\subset\bigcup_{k\geq0}\Delta_k.$$
The second assertion is now obvious.
\end{proof}

Finally, the third partition lemma deals with describing the order $\taill$ in terms of $\tail$ and $\leq$.

\begin{lem}\label{third partition lemma} Let $f,g\in L_0+L_1$ be such that $f = \mu(f)$, $g = \mu(g)$ and $g \taill f.$ There exists a collection $\{\Delta_k\}_{k\geq0}$ of pairwise disjoint sets such that
\begin{enumerate}[{\rm (i)}]
\item $f|_{\Delta_k}\head g|_{\Delta_k}$ for every $k\geq0;$
\item $g\leq f$ on the complement of $\bigcup_{k\geq0}\Delta_k;$
\end{enumerate}
\end{lem}
\begin{proof} Consider the set $\{g>f\}.$ Similarly to the previous proof, connected components of the set $\{g>f\}$ are intervals (closed or not) not reduced to points. Let us enumerate these intervals as $(a_k,b_k),$ $k\geq0$.

We have
$$\int_t^{\infty}(f-g)_+-\int_t^{\infty}(f-g)_-=\int_t^{\infty}(f-g)\geq0.$$
Let
$$H(t)=\sup\Big\{u:\ \int_u^{\infty}(f-g)_+=\int_t^{\infty}(f-g)_-\Big\}.$$
Obviously, $H$ is a monotone function, $H(t)\geq t$ for all $t>0$ and
$$\int_{H(t)}^{\infty}(f-g)_+=\int_t^{\infty}(f-g)_-.$$

Set
$$\Delta_k=(a_k,b_k)\cup\Big((H(a_k),H(b_k))\cap\{g\leq f\}\Big).$$

Note that
$$\int_{b_k}^{\infty}(f-g)_+=\int_{a_k}^{\infty}(f-g)_+\geq\int_{a_k}^{\infty}(f-g)_-$$
and, therefore, $H(a_k)\geq b_k.$

We claim that $\Delta_k\cap\Delta_l=\varnothing$ for $k\neq l.$ Indeed, let $a_k<b_k\leq a_l<b_l.$ We have $H(a_k)\leq H(b_k)\leq H(a_l)\leq H(b_l).$ Thus, $(H(a_k),H(b_k))\cap (H(a_l),H(b_l))=\varnothing.$ We now have
$$\Delta_k\cap\Delta_l=\Big((\Delta_k\cap\{f<g\})\cap(\Delta_l\cap\{f<g\})\Big)\bigcup\Big((\Delta_k\cap\{f\geq g\})\cap(\Delta_l\cap\{f\geq g\})\Big).$$
Obviously,
$$(\Delta_k\cap\{f<g\})\cap(\Delta_l\cap\{f<g\})=(a_k,b_k)\cap(a_l,b_l)=\varnothing,$$
$$(\Delta_k\cap\{f\geq g\})\cap(\Delta_l\cap\{f\geq g\})=(H(a_k),H(b_k))\cap (H(a_l),H(b_l))\cap\{f\geq g\}=\varnothing.$$
This proves the claim.

We now claim that
$$\int_{(t,\infty)\cap\Delta_k}(f-g)\geq0.$$
If $t\geq b_k,$ then, taking into account that $H(a_k)\geq b_k,$ we infer that $(t,\infty)\cap\Delta_k\subset\{f\geq g\}$ and the claim follows immediately. If $t\in(a_k,b_k),$ then
$$\int_{(t,\infty)\cap\Delta_k}(f-g)=\int_{(H(a_k),H(b_k))}(f-g)_+-\int_t^{b_k}(f-g)_-\geq$$
$$\geq \int_{(H(a_k),H(b_k))}(f-g)_+-\int_{a_k}^{b_k}(f-g)_-=0.$$
This proves the claim.

It follows from the claim and Scholium \ref{majorization fact} that $g\chi_{\Delta_k}\taill f\chi_{\Delta_k}.$ Since
$$\int_{\Delta_k}g=\int_{\Delta_k}f,$$
it follows that $g\chi_{\Delta_k}\tail f\chi_{\Delta_k},$ which immediately implies the first assertion.

By construction, $(a_k,b_k)\subset\Delta_k.$ Thus,
$$\{g>f\}=\bigcup_{k\geq0}(a_k,b_k)\subset\bigcup_{k\geq0}\Delta_k.$$
The second assertion is now obvious.
\end{proof}

\subsection{Construction of operators}

We repeat the same structure as in the previous subsection, proving four lemmas, each one dealing with a certain order : $\head$, $\tail$, $\headd$ and finally $\taill$.

\begin{lem}\label{first operator lemma} Let $p\in (0,\infty)$. Let $f,g\in L_p(0,\infty)$, assume that $|g|^p\head|f|^p$, $f=\mu(f)$ and $g=\mu(g)$. There exists a linear operator $T:\mathcal{X}(0,\infty)\to\mathcal{X}(0,\infty)$ such that $g=T(f)$ and
$$\|T\|_{L_p\to L_p}\leq 2\cdot 3^{\frac1p},\quad \|T\|_{L_{\infty}\to L_{\infty}}\leq 2\cdot 2^{\frac1p}.$$
\end{lem}
\begin{proof}

{\it Step 1:} First, let us assume that $f$ and $g$ are step functions

Apply Lemma \ref{first partition lemma} to the functions $f^p$ and $g^p$ and let $I_k$ and $J_k$ be as in Lemma \ref{first partition lemma}. Without loss of generality, the interval $I_k$ is located to the left of the interval $J_k.$

For every $k\geq0,$ define the mapping $S_k:\mathcal{X}(I_k\cup J_k)\to\mathcal{X}(I_k\cup J_k)$ as below. The construction of this mapping will depend on whether $f^p|_{J_k}\leq \frac12 g^p|_{J_k}$ or $f^p|_{J_k}>\frac12 g^p|_{J_k}.$

If $f^p|_{J_k}\leq \frac12 g^p|_{J_k},$ then
$$g^p|_{J_k}\cdot m(J_k)\leq g^p|_{I_k}\cdot m(I_k)+g^p|_{J_k}\cdot m(J_k)=$$
$$=f^p|_{I_k}\cdot m(I_k)+f^p|_{J_k}\cdot m(J_k)\leq f^p|_{I_k}\cdot m(I_k)+\frac12g^p|_{J_k}\cdot m(J_k).$$
Therefore,
$$g^p|_{J_k}\cdot m(J_k)\leq 2f^p|_{I_k}\cdot m(I_k).$$
Let $l_k$ be a linear bijection from $J_k$ to $I_k.$ We set
$$S_kx=\frac{g|_{I_k}}{f|_{I_k}}\cdot x\chi_{I_k}+\frac{g|_{J_k}}{f|_{I_k}}\cdot (x\circ l_k)\chi_{J_k}.$$
Clearly, $S_k$ is a contraction in the uniform norm.

Let $x\in L_p$. We have
$$\|S_kx\|_p^p\leq \frac{g^p|_{I_k}}{f^p|_{I_k}}\cdot \|x\chi_{I_k}\|_p^p+\frac{g^p|_{J_k}}{f^p|_{I_k}}\cdot\|(x\circ l_k)\chi_{J_k}\|_p^p\leq$$
$$\leq \frac{g^p|_{I_k}}{f^p|_{I_k}}\cdot \|x\|_p^p+\frac{g^p|_{J_k}}{f^p|_{I_k}}\cdot\frac{m(J_k)}{m(I_k)}\cdot\|x\|_p^p\leq 3\|x\|_p^p.$$
Also, we have
$$\|S_kx\|_{\infty}\leq \|x\|_{\infty}.$$

If $f^p|_{J_k}>\frac12 g^p|_{J_k},$ then we set $S_k=M_{gf^{-1}}.$ Clearly, $\|S_kx\|_{\infty}\leq 2^{\frac1p}\|x\|_{\infty}$ and $\|S_kx\|_p\leq 2^{\frac1p}\|x\|_p.$

We define $S : \X \to \X$ by:
$$S=\bigoplus_{k\geq0}S_k.$$
 Remark that for any $r\in [0,\infty]$, $\norm{S}_{r\to r} = \sup_{k\geq 0} \norm{S_k}_{L_r\to L_r}$. Hence,
$$\|S\|_{L_p\to L_p}\leq  3^{\frac1p},\quad \|S\|_{L_{\infty}\to L_{\infty}}\leq 2^{\frac1p}.$$
It remains only to note that $Sf=g.$

{\it Step 2:} Now, only assume that $f$ and $g$ are positive and non-increasing. Define for any $n\in\Zb$,
$$a_n = \sup\set{t \in (0,\infty) : f(t) \geq 2^{\frac{n}2}}\ \text{and}
\ b_n = \sup\set{t \in (0,\infty) : g(t) \geq 2^{\frac{n}2}}.$$
Let $\A$ be the $\sigma$-algebra generated by the intervals $(a_n,a_{n+1})$ and $(b_n,b_{n+1})$. Define
$$f_0^p = \Eb[f^p \mid \A]\ \text{and}\ g_0^p = \Eb[g^p \mid \A].$$
Note that $f_0$ and $g_0$ are step functions such that
$$g_0^p \head f_0^p,\quad 2^{-\frac12}f \leq f_0 \leq 2^{\frac12}f\quad \text{and}\quad 2^{-\frac12}g \leq g_0 \leq 2^{\frac12}g.$$
Apply Step 1 to $f_0$ and $g_0$ to obtain an operator $S$ and set
$$T = M_{ff_0^{-1}} \circ S \circ M_{g_0g^{-1}}.$$
Clearly, $Tf = g$ and
$$\norm{T}_{L_p \to L_p} \leq 2 \norm{S}_{L_p \to L_p} \leq 2\cdot 2^{\frac1p}\quad \text{ and }
\norm{T}_{L_\infty \to L_\infty} \leq 2 \norm{S}_{L_\infty \to L_\infty} \leq 2^2\cdot 2^{\frac1p}.$$

\end{proof}

\begin{lem}\label{second operator lemma} Let $f,g\in L_q(0,\infty)$ be positive non-increasing functions such that $g^q\tail f^q$. Let $d>1$. There exists a linear operator $T:\mathcal{X}(0,\infty)\to\mathcal{X}(0,\infty)$ such that $g=T(f)$ and
$$\|T\|_{L_0\to L_0}\leq 4,\quad \|T\|_{L_q\to L_q}\leq 2\cdot 3^{\frac1q}.$$
\end{lem}
\begin{proof}
Following step 2 of Lemma \ref{first operator lemma}, we are reduced to dealing with step functions.

Apply Lemma \ref{first partition lemma} to the functions $g^q$ and $f^q$ and let $(I_k)_{k\geq 1}$ and $(J_k)_{k\geq 1}$ be as in Lemma \ref{first partition lemma}. Without loss of generality, the intervals $I_k$ is located to the left of the intervals $J_k.$

Let $k\geq 1$. Define the mappings $S_k:\mathcal{X}(I_k\cup J_k)\to\mathcal{X}(I_k\cup J_k)$ as below.

Note that since $f^q|_{I_k\cup J_k}\head g^q|_{I_k\cup J_k},$
$$g|_{I_k}\geq f|_{I_k}\geq f|_{J_k} \geq g|_{J_k}.$$
 The construction of $S_k$ will depend on whether $\|f\chi_{I_k}\|_q\leq\|f\chi_{J_k}\|_q$ or $\|f\chi_{I_k}\|_q>\|f\chi_{J_k}\|_q.$

If $\|f\chi_{I_k}\|_q\leq\|f\chi_{J_k}\|_q,$ then $m(I_k)\leq m(J_k)$ and
$$g_0^q|_{I_k}\cdot m(I_k)\leq 2f^q|_{J_k}\cdot m(J_k).$$

Let $l_k:I_k\to J_k$ be a linear bijection. We set
$$S_kx=\frac{g|_{I_k}}{f|_{J_k}}(x\circ l_k)\chi_{I_k}+\frac{g}{f}x\chi_{J_k}.$$
Let $x \in L_q$.
$$\|S_kx\|_{0}\leq \|x\chi_{J_k}\|_0+\|(x\circ l_k)\chi_{I_k}\|_0\leq \p{1+\frac{m(I_k)}{m(J_k)}}\|x\|_0.$$
Thus, $\|S_k\|_{L_0\to L_0}\leq 2.$ We have
$$\|S_kx\|_q^q=\frac{g^q|_{I_k}}{f^q|_{J_k}}\|(x\circ l_k)\chi_{I_k}\|_q^q+\frac{g^q|_{J_k}}{f^q|_{J_k}}\|x\chi_{J_k}\|_q^q\leq$$
$$\leq \frac{g^q|_{I_k}}{f^q|_{J_k}}\cdot \frac{m(I_k)}{m(J_k)}\cdot\|x\|_q^q+\frac{g^q|_{J_k}}{f^q|_{J_k}}\|x\|_q^q\leq 3\|x\|_q^q.$$

If $\|f\chi_{I_k}\|_q>\|f\chi_{J_k}\|_q,$ then
$$g^q|_{I_k}\cdot m(I_k)\leq 2f^q|_{I_k}\cdot m(I_k)\mbox{ and, therefore, }g_0\leq 2^{\frac1q}f_0.$$
We set $S_k=M_{gf^{-1}}$. Obviously, $\|S_k\|_{L_0\to L_0}\leq 1$ and $\|S_k\|_{L_q\to L_q}\leq 2^{\frac1q}.$

We set
$$S=\bigoplus_{k\geq0}S_k.$$
Since  $S:\mathcal{X}(0,\infty)\to\mathcal{X}(0,\infty)$ is defined as a direct sum:
$$\|S\|_{L_0\to L_0} = \sup_{k\geq 1} \norm{S_k}_{L_0\to L_0} \leq 2,\quad \|S\|_{L_q\to L_q}= \sup_{k\geq 1} \norm{S_k}_{L_q\to L_q}\leq 3^{\frac1q}.$$

It remains only to note that $Sf=g.$
\end{proof}

\begin{lem}\label{third operator lemma} Let $f,g\in (L_p+L_{\infty})(0,\infty)$ be such that $|g|^p\headd|f|^p$, $f=\mu(f)$ and $g=\mu(g)$. There exists a linear operator $T:\mathcal{X}\to\mathcal{X}$ such that $g=T(f)$ and
$$\|T\|_{L_p\to L_p}\leq 2\cdot 3^{\frac1p},\quad \|T\|_{L_{\infty}\to L_{\infty}}\leq 2\cdot 2^{\frac1p}.$$
\end{lem}
\begin{proof}Let $(\Delta_k)_{k\geq0}$ be as in Lemma \ref{second partition lemma} and let $\Delta_{\infty}$ be the complement of $\bigcup_{k\geq0}\Delta_k.$ By Lemma \ref{first operator lemma}, there exists $T_k:\mathcal{X}(\Delta_k)\to\mathcal{X}(\Delta_k)$ such that $T_k(f)=g$ on $\Delta_k$ and such that
$$\|T_k\|_{L_p(\mathcal{X}_k)\to L_p(\mathcal{X}_k)}\leq 2\cdot 9^{\frac1p},\quad \|T_k\|_{L_{\infty}(\mathcal{X}_k)\to L_{\infty}(\mathcal{X}_k)}\leq 2\cdot 4^{\frac1p}.$$

Set $T_{\infty}=M_{\frac{g}{f}}$ on $\mathcal{X}(\Delta_{\infty}).$ We now set
$$T=T_{\infty}\bigoplus\Big(\bigoplus_{k\geq0}T_k\Big).$$
Obviously, $Tf=g$ on $(0,\infty)$ and
$$\|T\|_{L_p\to L_p}\leq 2\cdot 9^{\frac1p},\quad \|T\|_{L_{\infty}\to L_{\infty}}\leq 2\cdot 4^{\frac1p}.$$
\end{proof}

\begin{lem}\label{fourth operator lemma} Let $f,g\in (L_0+L_q)(0,\infty)$ be such that  $|g|^q \taill \md{f}^q$, $f=\mu(f)$ and $g=\mu(g)$. There exists a linear operator $T:\mathcal{X}\to\mathcal{X}$ such that $g=T(f)$ and
$$\|T\|_{L_0\to L_0}\leq 4,\quad \|T\|_{L_q\to L_q}\leq 2\cdot 3^{\frac1q}.$$
\end{lem}
\begin{proof} Without loss of generality, $g=\mu(g)$ and $f=\mu(f).$ Let $(\Delta_k)_{k\geq0}$ be as in Lemma \ref{third partition lemma} and let $\Delta_{\infty}$ be the complement of $\bigcup_{k\geq0}\Delta_k.$ By Lemma \ref{second operator lemma}, there exists $T_k:\mathcal{X}(\Delta_k)\to\mathcal{X}(\Delta_k)$ such that $T_k(f)=g$ on $\Delta_k$ and such that
$$\|T_k\|_{L_0(\mathcal{X}_k)\to L_0(\mathcal{X}_k)}\leq  8,\quad \|T_k\|_{L_q(\mathcal{X}_k)\to L_q(\mathcal{X}_k)}\leq 2\cdot 9^{\frac1q}.$$

Set $T_{\infty}=M_{\frac{g}{f}}$ on $\mathcal{X}(\Delta_{\infty}).$ We now set
$$T=T_{\infty}\bigoplus\Big(\bigoplus_{k\geq0}T_k\Big).$$
Obviously, $Tf=g$ on $(0,\infty)$ and
$$\|T\|_{L_0\to L_0}\leq  4,\quad \|T\|_{L_q\to L_q}\leq 2\cdot 3^{\frac1q}.$$
\end{proof}

\section{Interpolation spaces for the couple $(L_p,L_q)$}\label{section:main}

In this section, we obtain characterizations of interpolation spaces for the couple $(L_p,L_q)$, in terms of the majorization notions studied earlier. The necessity of the condition we consider is a direct consequence of the constructions explained in the previous section. The fact that it is sufficient is shown by linking majorization to the $K$-functional.

\begin{thm}\label{thm: function 1}
Let $0\leq p < q \leq \infty$. Let $E$ be a quasi-Banach function space such that $E\in{\rm Int}(L_p,L_q).$ There exist $c_{p,E}$ and $c_{q,E}$ in $\Rb_{>0}$ such that:
\begin{enumerate}[{\rm (i)}]
\item suppose $p\neq 0$: for any $f\in E$ and $g\in L_p+L_\infty$ if $|g|^p \headd |f|^p,$ then $g\in E$ and \mbox{$\|g\|_E\leq c_{p,E}\|f\|_E;$}
\item suppose $q\neq \infty$: for any $f\in E$ and $g\in L_0+L_q$ if $|g|^q \taill |f|^q,$ then $g\in E$ and $\|g\|_E\leq c_{q,E}\|f\|_E.$
\end{enumerate}
\end{thm}

\begin{proof} By Lemma \ref{interpolation implies symmetry} (for $p>0$) or Lemma \ref{interpolation implies symmetry second lemma} (for $p=0$), we may assume without loss of generality that $E$ is a symmetric function space.
       
Assume that $p\neq 0$. Let $f\in E$ and let $g\in L_p+L_\infty$ be such that $\md{g}^p \headd \md{f}^p.$ Since $E$ is symmetric, we may assume without loss of generality that $f=\mu(f)$ and $g=\mu(g).$  By Lemma \ref{third operator lemma}, there exists an operator $T$ such that $T(f) = g$ and
$$\|T\|_{(L_p,L_\infty) \to (L_p,L_\infty)}\leq 2\cdot 3^{\frac1p}.$$
Recall that $L_q$ is an interpolation space for the couple $(L_p,L_{\infty})$ (one can take, for example, real or complex interpolation method). Let $c_{p,q}$ be the interpolation constant of $L_q$ for the couple $(L_p,L_{\infty}).$ We have
$$\|T\|_{L_q\to L_q}\leq  c_{p,q}\cdot 2\cdot 3^{\frac1p}.$$

Let $c_E$ be the interpolation constant of $E$ for $(L_p,L_q)$ (see Theorem \ref{rem:def interpolation}). Then,
$$\|T\|_{E\to E}\leq c_E\cdot \max\{1,c_{p,q}\}\cdot 2\cdot 3^{\frac1p}.$$
Thus,
$$\|g\|_E\leq \|T\|_{E\to E}\|f\|_E\leq c_E\cdot \max\{1,c_{p,q}\}\cdot 2\cdot 3^{\frac1p}\|f\|_E.$$
This proves the first assertion. The proof of the second one follows {\it mutatis mutandi} using Lemma \ref{lem : Lp in Int L0 Lq} instead of complex interpolation and (for $p=0$) Theorem \ref{thm: inter L_0} instead of Theorem \ref{rem:def interpolation}.
\end{proof}

\begin{lem}\label{cad decomp lemma}
Assume that $0 < p < q < \infty$. Let $f,g \in L_p+L_q$ such that $f = \mu(f)$ and $g = \mu(g)$. Suppose that at every $t>0,$ one of the following inequalities holds
$$\int_0^t g^pds\leq\int_0^t f^pds\mbox{ or }\int_t^{\infty}g^qds\leq\int_t^{\infty}f^qds.$$
Then, there exist $g_1,g_2 \in (L_p+L_q)^+$ satisfying: $g=g_1+g_2,$ $g_1^p\headd f^p$ and $g_2^q\taill f^q.$
\end{lem}
\begin{proof} Set:
$$A=\Big\{t>0:\ \int_0^tg(s)^pds\leq\int_0^tf(s)^pds\Big\},$$
$$B=\Big\{t>0:\ \int_t^{\infty}g(s)^qds\leq\int_t^{\infty}f(s)^qds\Big\}.$$

Let
$$u_+(t)=\inf\{s\in A:\ s\geq t\},\quad u_-(t)=\sup\{s\in A: s\leq t\},$$
$$v_+(t)=\inf\{s\in B:\ s\geq t\},\quad v_-(t)=\sup\{s\in B: s\leq t\}.$$

\begin{sloppypar}
Note that, for $t\notin A,$ $f^p\chi_{(u_-(t),u_+(t))}\head g^p\chi_{(u_-(t),u_+(t))}$ and, therefore,
\mbox{$g(u_+(t)-0)\leq f(u_+(t)-0).$}
\end{sloppypar}

Set $h_1(t)=g(u_+(t)-0),$ $t>0.$ By definition, $u_+(t)\geq t$ for all $t>0.$ Since $g$ is decreasing, it follows that $h_1\leq g.$ Set $h_2=g\chi_B.$ Since $u_+(t)=t$ for $t\in A,$ it follows that $h_1=g$ on $A.$ Thus, $h_1+h_2\geq g\chi_A+g\chi_B\geq g.$

We claim that
$$\int_0^t\mu(s,h_1)^pds\leq\int_0^tf(s)^pds.$$

Indeed, for $t\in A,$ we have
$$\int_0^th_1(s)^pds\leq\int_0^tg(s)^pds\leq\int_0^tf(s)^pds.$$
For $t\notin A,$ we have $h_1(s)=g(u_+(t))$ for all $s\in (u_-(t),u_+(t)).$ Thus,
$$\int_0^th_1(s)^pds=\int_0^{u_-(t)}h_1(s)^pds+\int_{u_-(t)}^th_1(s)^pds\leq$$
$$\leq \int_0^{u_-(t)}g(s)^pds+\int_{u_-(t)}^tg(u_+(t))^pds.$$
Since
$$\int_0^{u_-(t)}g(s)^pds=\int_0^{u_-(t)}f(s)^pds,\quad g(u_+(t))\leq f(u_+(t)),$$
it follows that
$$\int_0^th_1(s)^pds\leq \int_0^{u_-(t)}f(s)^pds+\int_{u_-(t)}^tf(u_+(t))^pds\leq\int_0^tf(s)^pds.$$
Since $h_1=\mu(h_1),$ the claim follows.

We claim that
$$\int_t^{\infty}\mu(s,h_2)^qds\leq\int_t^{\infty}f(s)^qds.$$

For $t\in B,$ we have
$$\int_t^{\infty}h_2(s)^qds\leq\int_t^{\infty}g(s)^qds\leq\int_t^{\infty}f(s)^qds.$$
For $t\notin B,$ we have
$$\int_t^{\infty}h_2(s)^qds=\int_{v_+(t)}^{\infty}h_2(s)^qds\leq$$
$$\leq\int_{v_+(t)}^{\infty}g(s)^qds=\int_{v_+(t)}^{\infty}f(s)^qds\leq\int_t^{\infty}f(s)^qds.$$
In either case,
$$\int_t^{\infty}\mu(s,h_2)^qds\leq\int_t^{\infty}h_2(s)^qds\leq \int_t^{\infty}f(s)^qds.$$
This proves the claim.

Setting
$$g_1=\frac{h_1}{h_1+h_2}g,\quad g_2=\frac{h_2}{h_1+h_2}g,$$
we complete the proof.
\end{proof}

\begin{thm}\label{thm: monotonicity to inter}
 Let $0\leq p < q \leq \infty$ with either $p \neq 0$ or $q\neq \infty$. Let $E$ be a quasi-Banach function space. Assume that there exist $c_{p,E}$ and $c_{q,E}$ in $\Rb_{>0}$ such that:
\begin{enumerate}[{\rm (i)}]
\item if $p\neq 0$: for any $f\in E$ and $g\in L_0+L_p$, if $|g|^p \headd |f|^p,$ then $g\in E$ and \mbox{$\|g\|_E\leq c_{p,E}\|f\|_E;$}
\item if $q\neq \infty$: for any $f\in E$ and $g\in L_0+L_q$, if $|g|^q \taill |f|^q,$ then $g\in E$ and $\|g\|_E\leq c_{q,E}\|f\|_E;$
\end{enumerate}
Then $E$ belongs to ${\rm Int}(L_p,L_q).$
\end{thm}
\begin{proof} Assume that $p\neq 0$. Let us show that the first condition implies $E\subset L_p+L_{\infty}.$ Indeed, assume the contrary and choose $f\in E$ such that $\mu(f)\chi_{(0,1)}\notin L_p.$ Let
$$f_n=\min\{\mu(\frac1n,f),\mu(f)\chi_{(0,1)}\},\quad n\geq 1.$$
Obviously, $\|f_n\|_E\leq\|\mu(f)\chi_{(0,1)}\|_E\leq\|f\|_E.$ On the other hand, $\|f_n\|_p^p\chi_{(0,1)}\headd f_n^p.$ By the first condition on $E,$ we have $\|f_n\|_p\|\chi_{(0,1)}\|_E\leq c_{p,E}\|f\|_E.$ However $\|f_n\|_p\uparrow\|\mu(f)\chi_{(0,1)}\|_p=\infty.$ This contradiction shows that our initial assumption was incorrect. Thus, $E\subset L_p+L_{\infty}.$

A similar argument shows that the second condition implies $E\subset L_0+L_q.$ Thus, a combination of both conditions implies $E\subset L_p+L_q.$

Let $T$ be a contraction on $(L_p,L_q)$ and $f\in E$. To conclude the proof, it suffices to show that $Tf$ belongs to $E$. First, note that
$$K(t,Tf,L_p,L_q)\leq K(t,f,L_p,L_q).$$
Assume that $p>0$ and $q<\infty.$ Let $\alpha^{-1}=\frac1p-\frac1q.$ By Holmstedt formula for the $K$-functional (see \cite{Hol70}), there exists a constant $c_{p,q} >0$ such that for any $t\in\Rb_{>0}$:
$$\Big(\int_0^{t^{\alpha}}\mu(s,Tf)^pds\Big)^{\frac1p}+t\Big(\int_{t^{\alpha}}^{\infty}\mu(s,Tf)^qds\Big)^{\frac1q}\leq$$
$$\leq c_{p,q}\Big(\Big(\int_0^{t^{\alpha}}\mu(s,f)^pds\Big)^{\frac1p}+t\Big(\int_{t^{\alpha}}^{\infty}\mu(s,f)^qds\Big)^{\frac1q}\Big).$$
Hence, for any given $t>0,$ we either have
$$\int_0^{t^{\alpha}}\mu(s,Tf)^pds\leq\int_0^{t^{\alpha}}\mu(s,c_{p,q}f)^pds$$
or
$$\int_{t^{\alpha}}^{\infty}\mu(s,Tf)^qds\leq \int_{t^{\alpha}}^{\infty}\mu(s,c_{p,q}f)^qds.$$

By Lemma \ref{cad decomp lemma}, one can write
\begin{equation}
\mu(Tf)=g_1+g_2,\quad g_1^p\headd (c_{p,q}\mu(f))^p,\quad g_2 \taill (c_{p,q}\mu(f))^q.
\end{equation}
By assumption, we have
$$\|g_1\|_E\leq c_{p,E}\|f\|_E,\quad \|g_2\|_E\leq c_{q,E}\|f\|_E.$$
By triangle inequality, we have
$$\|Tf\|_E\leq c_{p,q,E}\|f\|_E.$$

Assume now that $p>0$ and $q = \infty$. This case is simpler since by the Holmstedt formula (see \cite{Hol70}), there exists $c_p\in \Rb_{>0}$ such that for any $t\in\Rb_{>0}$:
$$\Big(\int_0^{t^{p}}\mu(s,Tf)^pds\Big)^{\frac1p}\leq c_{p}\Big(\int_0^{t^{p}}\mu(s,f)^pds\Big)^{\frac1p}.$$
This means that $\md{Tf}^p \headd \md{c_pf}^p$ so by assumption (1), $Tf$ belongs to $E$ and
$$\|Tf\|_E\leq c_p c_{p,E}\|f\|_E.$$

The case of $p = 0$ and $q<\infty$ is given by Theorem \ref{cor: taill to int}.
\end{proof}

Theorem \ref{thm : intro inter} claimed in the introduction compiles some results of this section.

\begin{proof}[Proof of Theorem \ref{thm : intro inter}]
The assertion (a) is obtained by combining Theorem \ref{thm: function 1} and Theorem \ref{thm: monotonicity to inter} with $q = \infty$.

The assertion (b) is derived similarly from Theorem \ref{thm: function 1} and Theorem \ref{thm: monotonicity to inter} by applying them with $p = 0$.

Finally, using (a), (b), Theorem \ref{thm: function 1} and \ref{thm: monotonicity to inter}, for $0 < p < q < \infty$, one obtains (c).
\end{proof}

\begin{rem}
{\rm
In the spirit of Corollary \ref{cor: taill to int}, we could have used a non-quantitative condition to deal with the case of $q = \infty$ in Theorem \ref{thm: monotonicity to inter}. Let $E$ be a quasi-Banach function space and $p,q \in (0,\infty)$. This means that the two following conditions are equivalent:
\begin{enumerate}[{\rm (i)}]
\item for any $f\in E$, $g\in L_p + L_\infty$,
$$\md{g}^p \headd \md{f}^p \Rightarrow g \in E.$$
\item there exists $c>0$ such that for any $f\in E$, $g\in L_p + L_\infty$,
$$\md{g}^p \headd \md{f}^p \Rightarrow g \in E\ \text{and}\ \norm{g}_E \leq c\norm{f}_E.$$
\end{enumerate}
Similarly, the two following conditions are equivalent:
\begin{enumerate}[{\rm (i)}]
\item for any $f\in E$, $g\in L_0 + L_q$,
$$\md{g}^q \taill \md{f}^q \Rightarrow g \in E.$$
\item there exists $c>0$ such that for any $f\in E$, $g\in L_0 + L_q$,
$$\md{g}^q \taill \md{f}^q \Rightarrow g \in E\ \text{and}\ \norm{g}_E \leq c\norm{f}_E.$$
\end{enumerate}
}
\end{rem}

\section{Interrpolation spaces for couples of $\ell^p$-spaces} \label{sequence}

In this section, we show that our approach to the Lorentz-Shimogaki and Arazy-Cwikel theorems also applies to sequence spaces. We follow a structure similar to the previous sections, proving partition lemmas, then constructing bounded operators on couples $(\ell^p,\ell^q)$ with suitable properties to finally draw conclusions on the interpolation spaces of the couple $(\ell^p,\ell^q)$. Additional arguments involving Boyd indices will be required to prove Theorem \ref{thm:conj LSZ}. 

We identify sequences with bounded functions on $(0,\infty)$ which are almost constant on intervals of the form $(k,k+1)$, $k\in\Zb^+$ by
\begin{align*}
  i \colon \ell^\infty &\to L_\infty\\
  (u_k)_{k\in\Zb^+} &\mapsto \sum_{k=0}^{\infty} u_k\Ind_{(k,k+1)}.
\end{align*}

\subsection{An interpolation theorem for the couple $(\ell^p,\ell^q)$}

We start with a partition lemma playing, for sequence spaces, the role of Lemma \ref{third partition lemma}.

\begin{lem} \label{sequence partition lemma 1}
Let $a = (a_n)_{n\in\Zb^+}$, $b = (b_n)_{n\in\Zb^+}$ be two  positive decreasing sequences such that $b\taill a$. There exists a sequence $(\Delta_n)_{n\in\Zb^+}$ of subsets of $\Zb^+$ such that:
\begin{enumerate}[{\rm (i)}]
\item for every $k\in\mathbb{Z}_+,$ we have $\md{\set{n\in\Zb^+ : k\in \Delta_n}} \leq 3.$
\item $\sum_{k\in \Delta_n} a_k \geq b_n$ for any $n\in\Zb^+.$
\end{enumerate}
\end{lem}

\begin{proof}
Define $I = \set{n\in\Zb^+ : b_n > a_n}$. For $n\notin I$, set $\Delta_n = \set{n}$.

For any $n\in\Zb^+$, define:
$$i_n = \sup \set{i : \sum_{k=i}^{\infty} a_k \geq \sum_{k=n}^{\infty} b_k},$$
and for $n\in I$, $\Delta_n = \set{i_n,\dots,i_{n+1}}.$

From the definition of $i_n,$ we have
$$\sum_{k\geq i_n}a_k\geq\sum_{k\geq n}b_k.$$
From the definition of $i_{n+1},$ we have
$$\sum_{k>i_{n+1}}a_k<\sum_{k\geq n+1}b_k.$$
Taking the difference of these inequalities, we infer that
$$\sum_{k\in \Delta_n} a_k \geq b_n.$$
This proves the second condition.

Note that since $b \taill a,$ $i_n \geq n$ for any $n\in\Zb^+.$ Hence, if $n\in I,$ $b_n > a_n \geq a_{i_{n+1}}.$ Furthermore, by definition of $i_{n+1}$,
$$\sum_{k > i_{n+1}} a_k < \sum_{k = n+1}^{\infty} b_k\ \text{so} \sum_{k=i_{n+1}}^{\infty} a_k < \sum_{k=n}^{\infty}b_k.$$
Hence, by definition of $i_n$, $i_{n+1} > i_n$ for $n\in I.$

Let us now check the first condition. Suppose there exist distinct numbers $n_1,n_2,n_3\in I$ such that $k\in\Delta_{n_1},\Delta_{n_2},\Delta_{n_3}.$ Without loss of generality, $n_1<n_2<n_3.$ Since $k\in\Delta_{n_1},$ it follows that $k\leq i_{n_1+1}\leq i_{n_2}.$ Since $k\in\Delta_{n_3},$ it follows that $k\geq i_{n_3}\geq i_{n_2+1}.$ Hence, $i_{n_2+1}\leq k\leq i_{n_2}$ and, therefore, $i_{n_2+1}=i_{n_2}.$ Since $n_2\in I,$ it follows $i_{n_2+1}>i_{n_2}.$ This contradiction shows that $\md{\set{n\in I : k\in \Delta_n}} \leq 2.$ By definition, $k$ also belongs to at most one set $\Delta_n$, $n\notin I$. Consequently, $\md{\set{n\in\Zb^+ : k\in \Delta_n}} \leq 3.$
\end{proof}

From the partition lemma, we deduce an operator lemma similar to Lemma \ref{third operator lemma}. It extends Proposition 2 in \cite{Astashkin20}, which is established there for the special case $p=1$ by completely different method.

\begin{lem} \label{sequence operator lemma 1}
Let $p \geq 1$. Let $a,b \in \ell^p$ such that $\md{b}^p \taill \md{a}^p$. Then there exists an operator $T: \ell^p \to \ell^p$ such that:
\begin{enumerate}[{\rm (i)}]
\item $T(a) = b$.
\item $\norm{T}_{p\to p} \leq 3^{\frac1p}$ and $\norm{T}_{0\to 0} \leq 3.$
\end{enumerate}
\end{lem}

\begin{proof}
We can assume that both sequences are non-negative and decreasing. Apply Lemma \ref{sequence partition lemma 1} to $\md{a}^p$ and $\md{b}^p$. For every $n\in\Zb^+$, choose a linear form $\varphi_n$ on $\ell^p$ of norm less than $1$, supported on $\Delta_n$ and such that $\varphi_n(a) = \varphi_n(a\Ind_{\Delta_n}) = b_n$. Define
$$T: x \in \ell^p \mapsto (\varphi_n(x))_{n\in\Zb^+}.$$
By construction, $T(a) = b$. Let us check the norm estimates. Let $x\in\ell^p$,
\begin{align*}
\norm{T(x)}_p^p &= \sum_{n\in\Zb^+} \md{\varphi_n(x)}^p = \sum_{n\in\Zb^+} \md{\varphi_n(x\Ind_{\Delta_n)}}^p \\
&\leq \sum_{n\in\Zb^+} \sum_{k\in\Delta_n} \md{x_k}^p \\
&= \sum_{k\in\Zb^+} \md{\set{n:k\in\Delta_n}} \md{x_k}^p \\
&\leq 3\norm{x}_p^p.
\end{align*}
The second estimate is clear, using once again the fact that an integer $k$ belongs to at most three $\Delta_n$'s.
\end{proof}

The following remarks  were communicated to the authors by Cwikel and Nilsson.

\begin{rem} \label{rem : op on ellp to ell1}
 A bounded linear operator on $\ell^p$, $p\leq 1$ extends automatically to a bounded linear operator on $\ell^1.$
\end{rem}
\begin{proof}  Indeed, let $(e_n)_{n\in\mathbb{Z}^+}$ be the canonical basis of $\ell^\infty$. Let $T$ be a contraction on $\ell^p$, $p<1$. Then by H\"older's inequality $\norm{T(e_n)}_1 \leq \norm{T(e_n)}_p \leq\|T\|_{p\to p}.$ By the triangle inequality, for any finite sequence $a = (a_n)_{n\in\mathbb{Z}^+}$, $$\norm{T(a_n)}_1 \leq \sum_{n\in\mathbb{Z}^+} \md{a_n}\norm{T(e_n)}_1 \leq \norm{a}_1\|T\|_{p\to p}.$$
Hence, $T$ extends to a contraction on $\ell^1$.
\end{proof}

\begin{rem}
The condition $p\geq 1$ in Lemma \ref{sequence operator lemma 1} is necessary.
\end{rem}
\begin{proof} Let us show that Lemma \ref{sequence operator lemma 1} cannot be true for $p<1$. Assume by contradiction that there exists $c>0$ such that for any finite sequences $a$ and $b$ in $\ell^p$ such that $\md{b}^p \taill \md{a}^p$ there exists $T$ with $\norm{T}_{p\to p} \leq c$ and $T(a) = b$. By Remark \ref{rem : op on ellp to ell1} above, we also have $\norm{T}_{1\to 1} \leq c$. In particular, $\norm{b}_1 \leq c\norm{a}_1$. By considering $b = e_1$ and $a = \frac1{N^{1/p}}\sum_{i=1}^N e_i$ for $N$ large enough, one obtains a contradiction.
\end{proof}

We will not prove a sequence version of Lemma \ref{second operator lemma} to avoid the repetition of too many similar arguments. Fortunately, the expected result already appears in the literature, see \cite[Theorem 3]{Cwi81}:

\begin{lem}\label{sequence operator lemma 2}
Let $p > 0$. Let $a,b \in \ell^\infty$ such that $\md{b}^p \headd \md{a}^p$. Then there exists an operator $T: \ell^p \to \ell^p$ such that:
\begin{enumerate}[{\rm (i)}]
\item $T(a) = b$.
\item $\norm{T}_{p\to p} \leq 8^{1/p}$ and $\norm{T}_{\infty\to \infty} \leq 2^{1/p}$.
\end{enumerate}
\end{lem}

We conclude this subsection with a new interpolation theorem.

\begin{thm}\label{thm: sequence interpolation}
Let $p<q\in (0,\infty]$ such that $q\geq 1$. Let $E$ be a quasi-Banach sequence space. Then $E$ belongs to $\text{Int}(\ell^p,\ell^q)$ if and only if there exists $c_{p,E}$ and $c_{q,E}$ in $\Rb_{>0}$ such that:
\begin{enumerate}[{\rm (i)}]
\item for any $u\in E$ and $v\in\ell^\infty$, if $|v|^p\headd |u|^p,$ then $v\in E$ and $\|v\|_E\leq c_{p,E}\|u\|_E;$
\item if $q<\infty$ : for any $u\in E$ and $v\in\ell^\infty$, if $|v|^q\taill |u|^q,$ then $v\in E$ and $\|v\|_E\leq c_{q,E}\|u\|_E.$
\end{enumerate}
\end{thm}

\begin{proof}

The proof of the "only if" implication is identical to the proof of Theorem \ref{thm: function 1} using Lemmas \ref{sequence operator lemma 1} and \ref{sequence operator lemma 2} instead of Lemmas \ref{second operator lemma} and \ref{first operator lemma}. The "if" implication is given by \cite[Theorem 4.7]{Cad19}.

\end{proof}

\subsection{Upper Boyd index}

Let us now recall the definition of the upper Boyd index, in the case of sequence spaces. For any $n\in\Nb$ define the dilation operator:
\begin{align*}
  D_n \colon \ell^\infty &\to \ell^\infty\\
  (u_k)_{k\in\Zb^+} &\mapsto\p{ u_{\floor{\frac{k}{n}}}}_{k\in\Zb^+}.
\end{align*}
Let $E$ be a symmetric function space. Define the Boyd index associated to $E$ by:
$$\beta_E = \lim_{k\to\infty} \dfrac{\log \norm{D_k}_{E\to E}}{\log k}.$$
Note that since $E$ is a quasi-Banach space, $\beta_E < \infty$.

In this next proposition, we relate the upper Boyd index to an interpolation property. We follow \cite[Theorem 2]{Mon96}.

\begin{prop}\label{prop:Boyd}
Let $E$ be a quasi-Banach symmetric sequence space. Let $p<\frac{1}{\beta_E}$. There exists a constant $C$ such that for any $u\in E$ and $v\in\ell^\infty$, satisfying  $\md{v}^p \headd \md{u}^p,$ we have $v\in E$ and $\norm{v}_E \leq C \norm{u}_E$.
\end{prop}

Define the map $V:\ell_{\infty}\to\ell_{\infty}$ by setting
$$Vu=\sum_{n=0}^{\infty}2^{-n}D_{2^n}u,$$
and the map $C:\ell_{\infty}\to\ell_{\infty}$ by
$$(Cu)(n) = \dfrac1{n+1} \sum_{i=0}^n u_n.$$

\begin{lem}\label{V estimate lemma} If $p<\frac1{\beta_E},$ then
$$\big\|(V(u^p))^{\frac1p}\big\|_E \leq c_{p,E}\|u\|_E,\quad 0\leq u\in E.$$
\end{lem}
\begin{proof} Let $E_p$ be the $p-$concavification of $E,$ that is,
$$E_p=\{f:\ |f|^{\frac1p}\in E\},\quad \|f\|_{E_p}=\||f|^{\frac1p}\|_E^p.$$
Obviously, $E_p$ is a quasi-Banach space. Apply Aoki-Rolewicz theorem to the space $E_p$ and fix $q=q_{p,E}>0$ such that
$$\|\sum_{n\geq0}x_n\|_{E_p}^q\leq C_{p,E}\sum_{n\geq0}\|x_n\|_{E_p}^q.$$
       
For every $u\in E,$ we have
$$\big\|(V(u^p))^{\frac1p}\big\|_E^{qp}=\|V(u^p)\|_{E_p}^q=\|\sum_{n=0}^{\infty} \frac{1}{2^n} (D_{2^{n}}u)^p\|_{E_p}^q\leq$$
$$\leq C_{p,E}\sum_{n=0}^{\infty}\|\frac{1}{2^n} (D_{2^{n}}u)^p\|_{E_p}^q=C_{p,E}\sum_{n=0}^{\infty}2^{-nq}\|D_{2^n}u\|_E^{qp}.$$
       
Let $r\in(p,\beta_E^{-1}).$ By the definition of $\beta_E$, there exists $c_{p,E}>0$ such that $\norm{D_n}_{E\to E} \leq c_{p,E}n^{\frac1r}$ for any $n\in\Nb.$ Therefore,
$$\big\|(V(u^p))^{\frac1p}\big\|_E^{qp}\leq C_{p,E}\cdot c_{p,E}^q\cdot \sum_{n=0}^{\infty}2^{-nq}2^{\frac{nqp}{r}}\|u\|_E^{qp}=C_{p,E}\cdot c_{p,E}^q\cdot \frac{2^q}{2^q-2^{\frac{qp}{r}}}\cdot \|u\|_E^{qp}.$$
\end{proof}

\begin{lem}\label{C V comparison} If $x=\mu(x),$ then $Cx\leq 3Vx$ for every $x\in\ell_{\infty}.$
\end{lem}
\begin{proof} Let $k \geq 0$. Since $x$ is decreasing, it follows that
$$(Cx)(2^k-1)=\frac{1}{2^k}\Big(x(0)+ \sum_{i=0}^{k-1}\sum_{j=2^i}^{2^{i+1}-1}x(j)
\Big)\leq \frac{1}{2^k}\Big(x(0) + \sum_{i=0}^{k-1}2^ix(2^i)\Big).$$
On the other hand, we have
$$(Vx)(2^{k+1}-1)=\sum_{n\geq0}2^{-n}x\big(\floor{\frac{2^{k+1}-1}{2^n}}\big)=$$
$$=\sum_{n=0}^k2^{-n}x(2^{k+1-n}-1)+\sum_{n=k+1}^{\infty}2^{-n}x(0)=$$
$$=\frac{1}{2^k}\Big(x(0)+\sum_{i=0}^k2^ix(2^{i+1}-1)\Big).$$
Again using the fact that $x$ is decreasing, we obtain
$$\sum_{i=0}^{k-1}2^ix(2^i)=x(1)+\sum_{i=0}^{k-2}2^{i+1}x(2^{i+1})\leq$$
$$\leq x(1)+2\sum_{i=0}^{k-2}2^ix(2^{i+1}-1)\leq 3\sum_{i=0}^k2^ix(2^{i+1}-1).$$

Combining the three previous inequalities, we have just shown that for any $k\geq 0$,
$$(Cx)(2^k - 1) \leq 3 (Vx)(2^{k+1} - 1).$$
Now let $n\geq 0$ and choose $k$ such that $n \in [2^k - 1, 2^{k+1} - 1]$.
Since $Cx$ and $Vx$ are decreasing, we have:
$$(Cx)(n) \leq (Cx)(2^k - 1) \leq 3(Vx)(2^{k+1} - 1) \leq 3(Vx)(n).$$

\end{proof}

\begin{proof}[Proof of Proposition \ref{prop:Boyd}] Without loss of generality, $u=\mu(u)$ and $v=\mu(v).$ Since $v^p\headd u^p,$ it follows that
$$|v|^p\leq C(|v|^p)\leq C(|u|^p)\leq 3 V(|u|^p),$$
where we used Lemma \ref{C V comparison} to obtain the last inequality. By Lemma \ref{V estimate lemma}, we have
$$\|v\|_E\leq 3^{\frac1p}\big\|(V(|u|^p))^{\frac1p}\big\|_E\leq c_{p,E}\|u\|_p.$$
\end{proof}

We are now ready to deliver a complete resolution of the conjecture stated by Levitina and two of the authors in \cite{LSZ20}.

\begin{proof}[Proof of Theorem \ref{thm:conj LSZ}] Let $E$ be a quasi-Banach sequence space. Let $q\geq 1$. Recall that Theorem \ref{thm:conj LSZ} states that the two following conditions are equivalent:
\begin{enumerate}[{\rm (a)}]
\item there exists $p<q$ such that $E$ is an interpolation space for the couple $(\ell^p,\ell^q)$;
\item there exists $c>0$ such that for any $u\in E$ and $|v|^q \taill |u|^q,$ then $v\in E$ and $\|v\|_E\leq c\|u\|_E.$
\end{enumerate}
Note that by Lemma \ref{interpolation implies symmetry}, we may assume that $E$ is a symmetric space.

$(a) \Rightarrow (b)$. This is immediate by Theorem \ref{thm: sequence interpolation}.

$(b) \Rightarrow (a)$. Let $p < 1/\beta_E$. By Proposition \ref{prop:Boyd}, for any sequence $u\in E$ and $v\in \ell^\infty$, if $\md{v}^p \headd \md{u}^p$, $v \in E$ and $\norm{v}_E \leq c_{p,E}\norm{u}_E$. Applying Theorem \ref{thm: sequence interpolation} for indices $p$ and $q$, we obtain that $E$ belongs to $\text{Int}(\ell^p,\ell^q)$.
\end{proof}

\bibliographystyle{plain}
\bibliography{CSZ}

\end{document}